\documentclass[11pt]{article}

\usepackage[utf8]{inputenc}
\usepackage{emptypage}
\usepackage{amsmath}
\usepackage{amsthm}
\usepackage{amsfonts}
\usepackage{amssymb}
\usepackage{mathrsfs}
\usepackage{graphicx}
\usepackage{color}  
\usepackage{bbm}
\usepackage{enumerate}
\usepackage{verbatim}

\newtheorem{theorem}{Theorem}[section]

\theoremstyle{theorem}
\newtheorem{corollary}[theorem]{Corollary}
\theoremstyle{theorem}
\newtheorem{lemma}[theorem]{Lemma}
\theoremstyle{definition}

\theoremstyle{theorem}

\theoremstyle{theorem}

\theoremstyle{theorem}

\theoremstyle{theorem}

\newcommand{\Z}{\mathbb Z}
\newcommand{\N}{\mathbb N}
\newcommand{\B}{\mathcal B}
\newcommand{\Pp}{\mathbb P}
\newcommand{\E}{\mathbb E}
\newcommand{\s}{\mathfrak s}
\newcommand{\mvec}{\mathbf m}
\newcommand{\svec}{\mathbf s}
\newcommand{\phivec}{\boldsymbol{\phi}}
\newcommand{\M}{\mathcal M}
\newcommand{\midcap}{{\,\textstyle \bigcap\,}}

\newcommand{\infnorm}[1]{\left\|#1\right\|_\infty}

\title{\bf Active Phase for Activated Random Walk on $\Z^2$}
\date{}

\author{Yiping Hu\thanks{University of Washington, Seattle, USA; huypken@uw.edu}}
\begin{document}

\maketitle

\begin{abstract}
	We show that for small enough sleep rate, the critical density of the symmetric Activated Random Walk model on $\Z^2$ is strictly less than one.
\end{abstract}

\section{Introduction}

The Activated Random Walk (ARW) model is a system of interacting particles, which is believed to exhibit self-organized criticality. In this paper, we study the dynamics of the simple ARW on the square lattice $\Z^2$. The process starts with a random number of particles at each site $x \in \Z^2$ distributed as a spatially ergodic distribution with average density $\zeta>0$. Each particle can be in one of two states, either active or sleepy. Initially all particles are active. Every active particle performs a continuous-time simple random walk at rate 1 until it falls asleep, which happens with sleep rate $\lambda$. A sleepy particle stays put until it becomes active, i.e. it gets reactivated whenever an active particle arrives at its site.

For every sleep rate $\lambda$, the system has an \textit{absorbing-state phase transition}: if the initial density $\zeta$ of particles is below a critical value $\zeta_c(\lambda)$, all particles eventually sleep and the system fixates; above $\zeta_c(\lambda)$, each particle moves indefinitely and the system stays active. It is intuitively clear and was proved in \cite{GA10, Shellef10} that $ \zeta_c(\lambda)\le 1$ for all $\lambda$ in all dimensions.

Our main result answers an open question about the existence of a non-trivial active phase on $\Z^2$.

\begin{theorem} \label{thm:active}
There is a constant $C>0$ such that for all $\lambda<1$, $$ \zeta_c(\lambda) \le \frac{C}{\ln(1/\lambda)}.$$ In other words, the system stays active a.s. if $\zeta > \frac{C}{\ln(1/\lambda)}$.
\end{theorem}

\begin{corollary}
\label{cor:small lambda}
For small enough $\lambda$, the critical density $\zeta_c(\lambda)<1$.
\end{corollary}

\begin{corollary}
\label{cor:go to zero}
The critical density $\zeta_c(\lambda)\to 0$ as $\lambda \to 0$.	
\end{corollary}

Though we choose to demonstrate our proof in $\Z^2$, the same argument works in all dimensions.

In one dimension, similar results as Corollary \ref{cor:small lambda} and \ref{cor:go to zero} were first proved in the breakthrough work of \cite{BGH18}. The approach from \cite{BGH18} was later reformulated in \cite{Leo20} as based on two major ingredients: a mass balance equation between blocks and a single-block estimate. More recently, this reformulation has been adopted successfully in \cite{ARS19,HRR20} to find the right order of asymptotics in Corollary \ref{cor:go to zero} and extend Corollary \ref{cor:small lambda} for all $\lambda>0$ respectively. All these works rely crucially on the topological convenience of $\Z$ in order to maintain independence structures between blocks. However, we no longer enjoy such property in $\Z^2$, making the block arguments rather unwieldy. Similar results were also obtained in higher dimensions $d\ge3$ \cite{ST18, Taggi19} or for biased ARW \cite{Taggi16, RT18}, but these arguments only work for transient walks. See \cite{Leo20} for a more complete survey of results.

In this paper, we take a step back and consider the mass balance equations between \textit{sites} instead of blocks. A straightforward energy-entropy calculation is then employed. Our main goal is to show that a similar line of arguments works in all dimensions and yields non-trivial results. The argument also provides simpler proofs and a handy framework for further understanding.

We note that another proof of the non-trivial active phase in two dimensions was obtained independently by Forien and Gaudilli\` ere \cite{FG22}. The asymptotic upper bound in Theorem \ref{thm:active} is comparable to the one from \cite[Theorem 2]{FG22}.

In order to establish Theorem \ref{thm:active}, it suffices to prove the following quantitative estimate for the \textit{finite-volume dynamics}. For any positive integer $N$, let $\mathcal B_N:= \{-N,\dots,N\}^2$ be the finite box in $\Z^2$ and let $|\B_N|$ be the cardinality of $\B_N$. Consider the ARW dynamics on $\mathcal B_N$ where the walks are killed upon leaving $\mathcal B_N$. Without loss of generality, we only consider the initial configurations with at most $|\B_N|$ particles in $B_N$. The finite-volume dynamics eventually stabilizes, that is when all particles that remain in $\B_N$ are sleeping. We call $S(\B_N)$ the number of sleeping particles in $\B_N$ after stabilization.

\begin{theorem}\label{thm:finite volume}
	There exist positive constants $C$ and $c$, such that for any integer $N$ and any initial configuration in $\mathcal B_N$,
	$$\Pp\left(S(\B_N) \ge \frac{C}{\ln(1/\lambda)}|\B_N|\right) \le e^{-cN^2}.$$
\end{theorem}

Theorem \ref{thm:finite volume}, combined with the criterion \cite[Theorem 2.11]{Leo20} and the universality result \cite{RSZ19}, gives a proof of Theorem \ref{thm:active}. The rest of the paper is devoted to the proof of \ref{thm:finite volume}.

\section{Modified site-wise representation}
\label{sec:site-wise}

Formally, the configuration of the system at time $t \ge 0$ is given by $\eta_t \in \{0,\s,1,2,\dots\}^{\Z^2}$, where $\s$ represents a single sleeping particle and we have $0 < \s < 1$. Write $x \sim y$ if $x$ is a neighbor of $y$. In the continuous-time ARW, for $x \in \Z^2$ such that $\eta_{t-}(x)\ge 1$, the system undergoes transitions $\eta_t = \mathfrak t_{x,y} \eta_{t-}$ and $\eta_t = \mathfrak t_{x,\s} \eta_{t-}$ at rates $\frac{1}{4}\eta_{t-}(x)$ and $\lambda\eta_{t-}(x)$ respectively. For $x \sim y$, the \textit{movement} transition $\mathfrak t_{x,y}$ is defined by $$\mathfrak t_{x,y}(\eta)(z)=\begin{cases}
\eta(y)+1,&z=y,\\
\eta(x)-1,&z=x,\\
\eta(z),&\text{otherwise},	
\end{cases}
$$
where we use the convention $\s+1=2$. The \textit{sleeping} instruction is given by $$
\mathfrak t_{x,\s}(\eta)(z)=\begin{cases}
\s,&z=x \text{ and }\eta(x)=1,\\
\eta(z),&\text{otherwise.}	
\end{cases}
$$

It is useful to consider the site-wise representation of ARW. However, we will use a slight modification of the usual site-wise representation, cf.\cite{Leo20}. In our setting, each site $x\in \Z^2$ is associated with one stack of i.i.d. movement instructions $\left( \xi_k^x\right)_{k\in\N^+}$, together with another independent stack of i.i.d. geometric random variables $\left( g_k^x\right)_{k\in\N}$. More specifically, each movement $\xi_k^x$ has the distribution of $\mathfrak t_{x,y_k}$, where $y_k$ is a uniformly random neighbor of $x$, while each geometric $g_k^x$ encodes the number of sleep instructions between consecutive movements $\xi_k^x$ and $\xi_{k+1}^x$, thus taking value in $\{0,1,2,\dots\}$ with success probability $1/(1+\lambda)$.

It has been proved \cite[Section 11.2]{Leo20} that the stacks of instructions $\xi$ and $g$, with poisson clocks attached to every site, gives an explicit construction of the ARW process $(\eta_t)_{t\ge0}$ on $\Z^2$. We will not need the Abelian property of the ARW in the remaining sections, but note that essential inputs of Theorem \ref{thm:active}, including \cite[Theorem 2.11]{Leo20}, \cite{RSZ19} and \cite[Section 11.2]{Leo20}, all use the Abelian property in their proofs.

For $x \sim y$, let $n_{x,y}(m)$ be the number of times that $\mathfrak t_{x,y}$ appears in the first $m$ movement instructions $\left( \xi_k^x\right)_{k=1}^m$. Also let $\chi_x(m):=1\{g_m^x>0\}$ be the Bernoulli random variable with parameter $\lambda/(1+\lambda)$.

\section{Mass balance equation}

In this section we focus on the finite-volume dynamics and use the mass balance equation to prove Theorem \ref{thm:finite volume}.

We start by introducing three \textit{random} fields $\mathbf M=(M_x)\in\N^{\B_N}$, $\mathbf S=(S_x)\in\{0,1\}^{\B_N}$ and $\boldsymbol \Phi=(\Phi_x)\in\N^{\partial\B_N}$. First, define the \textit{activity odometer field} $M_x$ which counts the number of movement instructions used at $x$ until stabilization. Note that in $M_x$, we do \textit{not} count any sleep instruction. Secondly, let $S_x$ be the indicator random variable of there being a sleeping particle at $x$ after stabilization. Thirdly, define the \textit{exit measure} $\Phi_x$ to be the number of particles killed upon entering a site $x\in \partial B_N$. Here, for any $A \subseteq \Z^2$, we define the boundary set $\partial A:= \{x\in A^c; \text{there exists } y \in A \text{ such that } y \sim x\}$.

Now suppose we're given \textit{deterministic} fields $\mvec=(m_x)\in\N^{\B_N}$, $\svec=(s_x)\in\{0,1\}^{\B_N}$ and $\phivec=(\phi_x)\in\N^{\partial\B_N}$. A tuple of fields $(\mvec,\svec,\phivec)$ is said to satisfy the \textit{mass balance equation} if
\begin{enumerate}[(i)]
\item the mass balance equation holds at every $x \in \B_N$:
\begin{equation} \label{eq:mass balance}
\eta_0(x) + \sum_{y \sim x} n_{y,x}(m_y) = m_x + s_x,
\end{equation}
where $\eta_0$ denotes the initial configuration;
\item the \textit{boundary condition} is met at every $x\in\partial\B_N$:
\begin{equation}\label{eq:boundary condition}
	n_{R(x),x}(m_{R(x)}) = \phi_x,
\end{equation}
where $R(x)$ is the unique neighbor of $x$ contained in $\B_N$.
\end{enumerate}
Note that the tuple of random fields $(\mathbf M, \mathbf S, \boldsymbol \Phi)$ defined above always satisfies the mass balance equation.

To prove Theorem \ref{thm:finite volume}, we rely on the following entropy bound on all possible field configurations satisfying the mass balance equation. Denote by $\M(x,\mvec,\svec,\phivec)$ the event where equation (\ref{eq:mass balance}) holds at $x\in\B_N$; for $x\in\partial\B_N$, we use the same notation for the event where equation (\ref{eq:boundary condition}) holds at $x$. Also note that $\infnorm{\boldsymbol \Phi}\le|\eta_0|$, where $|\eta_0|:=\sum_{x\in\B_N}\eta_0(x)$.

\begin{lemma}
\label{lem:entropy bound}
There exists $c_1>0$ such that for any $N\in\N^+$ and initial configuration $
\eta_0$ in $\mathcal B_N$ with $|\eta_0|\le|B_N|$,
\begin{equation}
\label{eq:entropy bound}
\sum_{\mvec, \svec, \phivec}\Pp\Bigg(\bigcap\limits_{x\in\B_N \cup \partial\B_N} \M(x,\mvec,\svec,\phivec)\Bigg) \le e^{c_1 |\B_N|},	
\end{equation}
where the summation is over $\infnorm\phivec\le |\eta_0|$ and arbitrary $\mvec,\svec$.
\end{lemma}

The proof of Lemma \ref{lem:entropy bound} will be given in the next section.

\begin{proof}[Proof of Theorem \ref{thm:finite volume}]
Recall from Theorem \ref{thm:finite volume} that $S(\B_N)$ represents the number of sleeping particles in $\B_N$ after stabilization, so we have $S(\B_N)=\sum_{x \in\B_N} S_x$. Note also that \begin{equation*}
	\label{eq:S_x upper bound}
	S_x\le\chi_x(M_x), 	
 	\end{equation*}
where $\chi_x(m)$ is defined in Section \ref{sec:site-wise}.
Thus in order to prove Theorem \ref{thm:finite volume}, it suffices to upper bound $\Pp(\mathcal A(\mathbf M))$, where $\mathbf M$ is the odometer field, and for a fixed $\mvec$, we use $\mathcal A(\mvec)$ to denote the event that $$ \sum_{x\in\B_N} \chi_x(m_x) \ge \frac{C}{\ln(1/\lambda)}|\B_N|.$$ By decomposing over the fields we get $\Pp(\mathcal A(\mathbf M))$ is equal to
\begin{align}
&\sum_{\mvec, \svec, \phivec}\Pp\Bigg(\mathcal A(\mathbf m) \midcap \{(\mvec,\svec,\phivec) = (\mathbf M, \mathbf S, \boldsymbol \Phi)\} \midcap \bigcap\limits_{x\in\B_N \cup \partial\B_N} \M(x,\mvec,\svec,\phivec)\Bigg) \nonumber\\
&\le \sum_{\mvec, \svec, \phivec}\Pp\Bigg(\mathcal A(\mathbf m) \midcap  \bigcap\limits_{x\in\B_N \cup \partial\B_N} \M(x,\mvec,\svec,\phivec)\Bigg) \nonumber\\
&\le \sum_{\mvec, \svec, \phivec}\Pp\left(\mathcal A(\mathbf m)\right)  \Pp\Bigg(\bigcap\limits_{x\in\B_N \cup \partial\B_N} \M(x,\mvec,\svec,\phivec)\Bigg), \label{eq:independence}
\end{align}
where in the last inequality we use the independence between $\chi_x(m)$'s and $n_{x,y}(m)$'s. For any fixed $\mvec$, the probability of $\mathcal A(\mvec)$ is controlled via Chernoff bound by
\begin{equation}
\label{eq:Chernoff bound}
e^{-D_{\text{KL}}\left(\frac{C}{\ln(1/\lambda)} \big{\|} \frac{\lambda}{1+\lambda}\right)\,|\B_N|}
\end{equation}
where $$D_{\text{KL}}(p_1\|\,p_2)=p_1\ln \frac{p_1}{p_2}+(1-p_1)\ln \left(\frac{1-p_1}{1-p_2}\right)$$ is the Kullback-Leibler divergence between Bernoulli distributed random variables with parameters $p_1$ and $p_2$ respectively.
By combining (\ref{eq:independence}), (\ref{eq:Chernoff bound}) and Lemma \ref{lem:entropy bound} and picking a large enough $C$, we complete the proof of Theorem \ref{thm:finite volume}.
\end{proof}

\section{Entropy bound}
In this section we prove the entropy bound in Lemma \ref{lem:entropy bound}. Partition $\B_N$ into a set of cycles $\B_N=\coprod_{n=0}^N C_n$, where each $C_n:=\{x \in \Z^2; \infnorm{x}=N-n\}$. Set $C_{-1}:=\partial\B_N$. Recall from Section \ref{sec:site-wise} that there is a stack of movement instruction $(\xi_k^x)_{k\in\N^+}$ at every $x\in\B_N$. Consider the filtration $(\mathcal F_n)_{n=-1}^N$, where $\mathcal F_{-1}$ is the trivial $\sigma$-field and for $n \ge 0$, each $\mathcal F_n$ is the $\sigma$-field generated by the instructions $\{\xi_k^x; k\in\N^+,x\in C_i,0\le i \le n\}$ in the outermost $n+1$ cycles.

\begin{lemma}
\label{lem:conditional entropy bound}
Fix $N, \eta_0, \svec$ and $\phivec$. For every $n=0,\dots,N$, a.s.,
\begin{equation}
\label{eq:conditional entropy bound}
\E\bigg[\,\sum_{\mvec_n}\prod_{x \in C_{n-1}}1_{\mathcal M(x,\mvec,\svec,\phivec)}\,\bigg\vert\, \mathcal F_{n-1}\bigg] \le 4^{|C_n|},
\end{equation}
where $\mvec_n:=(m_x)_{x\in C_n}$ is the restriction of $\mvec$ to $C_n$, and the summation is over all $\mvec_n\in\N^{C_n}$.
\end{lemma}

\begin{proof}
Define the backward neighbor function $B$ that maps each $y\in C_n$ to a neighbor of $y$ in $C_{n-1}$; whenever there are multiple such neighbors, just pick one arbitrarily. Note that $B$ is a bijection from $C_n$ to $B(C_n)$. Rather than working with (\ref{eq:conditional entropy bound}), it suffices to give an upper bound of a slightly larger expression
\begin{equation}
\label{eq:upper}
\E\bigg[\,\sum_{\mvec_n}\prod_{y \in C_n}1_{\mathcal M(B(y),\mvec,\svec,\phivec)}\,\bigg\vert\, \mathcal F_{n-1}\bigg].
\end{equation}
By either the mass balance equation (\ref{eq:mass balance}) when $n\ge 1$ or the  boundary condition (\ref{eq:boundary condition}) when $n=0$, the event $\mathcal M(B(y),\mvec,\svec,\phivec)$ above depends on $\mvec_n$ only through $m_y$. Thus we may rewrite (\ref{eq:upper}) as
\begin{equation}
\label{eq:exchange}
\E\bigg[\,\prod_{y \in C_n} \sum_{m_y}1_{\mathcal M(B(y),\mvec,\svec,\phivec)}\,\bigg\vert\, \mathcal F_{n-1}\bigg].
\end{equation}
Now note that all randomness defining $\mathcal M(B(y),\mvec,\svec,\phivec)$, except for $n_{y,B(y)}(m_y)$, are measurable with respect to $\mathcal F_{n-1}$. Therefore, the sums over $m_y$ in (\ref{eq:exchange}), conditioned on $\mathcal F_{n-1}$, are stochastically dominated by i.i.d. geometric random variables with parameter $1/4$. This implies that both (\ref{eq:exchange}) and (\ref{eq:upper}) are controlled by $4^{|C_n|}$, thus proving (\ref{eq:conditional entropy bound}).
\end{proof}

\begin{proof}[Proof of Lemma \ref{lem:entropy bound}]
We will inductively show that for $0\le n \le N$,
\begin{equation}
\label{eq:induction}
\E\bigg[\,\sum_{\mvec_0}\sum_{\mvec_1}\cdots\sum_{\mvec_n} \prod_{x\in A_{n-1}}1_{\mathcal M(x,\mvec,\svec,\phivec)}\bigg] \le 4^{|A_n\setminus A_{-1}|},
\end{equation}
where $A_n:=\bigcup_{i=-1}^n C_i$. Assuming (\ref{eq:induction}) for $n=N$, it follows that the left-hand side of (\ref{eq:entropy bound}) is bounded above by $4^{|\B_N|}\cdot 2^{|\B_N|}\cdot |\B_N|^{|\partial\B_N|}$. This would imply Lemma \ref{lem:entropy bound} by choosing $c_1$ large enough.

It remains to prove (\ref{eq:induction}) by induction. The base case of $n=0$ is nothing but Lemma \ref{lem:conditional entropy bound} when $n=0$. Suppose the inequality is true for $n-1$. By conditioning on $\mathcal F_{n-1}$ and using Lemma \ref{lem:conditional entropy bound}, we get
\begin{align*}
\E\bigg[\,\sum_{\mvec_0}&\sum_{\mvec_1}\cdots\sum_{\mvec_{n}} \prod_{x\in A_{n-1}}1_{\mathcal M(x,\mvec,\svec,\phivec)}\bigg]\\
&= \E\bigg[\,\sum_{\mvec_0}\sum_{\mvec_1}\cdots\sum_{\mvec_{n-1}} \prod_{x\in A_{n-2}}1_{\mathcal M(x,\mvec,\svec,\phivec)}\times\\
&\hspace{.7in} \E\bigg[\,\sum_{\mvec_n}\prod_{x \in C_{n-1}}1_{\mathcal M(x,\mvec,\svec,\phivec)}\,\bigg\vert\, \mathcal F_{n-1}\bigg] \bigg]\\
&\le 4^{|C_n|} \cdot \E\bigg[\,\sum_{\mvec_0}\sum_{\mvec_1}\cdots\sum_{\mvec_{n-1}} \prod_{x\in A_{n-2}}1_{\mathcal M(x,\mvec,\svec,\phivec)}\bigg]\\
&\le 4^{|C_n|}\cdot 4^{|A_{n-1}\setminus A_{-1}|} = 4^{|A_n \setminus A_{-1}|}.
\end{align*}
\end{proof}

\subsection*{Acknowledgement}

The author would like to thank Christopher Hoffman for many useful comments and helpful conversations. Y.H. was supported by NSF grant DMS-1444084. 

\bibliography{ARW}{} 
\bibliographystyle{plain}
\end{document}